\documentclass[11pt]{amsart}

\usepackage{amsmath,amssymb,amsthm,amscd,amsfonts}
\usepackage{mathrsfs}
\usepackage{bbm}
\usepackage{enumerate}
\usepackage{graphicx}

\newtheoremstyle{morespacestyle}
  {5mm}
  {5mm}
  {\it}
  {0pt}
  {\bfseries}
  {.}
  {2mm}
  {}

\newtheoremstyle{nonitalicstyle}
  {5mm}
  {5mm}
  {}
  {0pt}
  {\bfseries}
  {.}
  {2mm}
  {}

\newtheorem{theorem}{Theorem}[section]
\newtheorem{lemma}[theorem]{Lemma}

\newtheorem{proposition}[theorem]{Proposition}
\newtheorem{corollary}[theorem]{Corollary}

\theoremstyle{nonitalicstyle}

\newtheorem{remark}[theorem]{Remark}

\numberwithin{equation}{section}

\newcommand{\abs}[1]{\left| \hspace{1pt} #1 \hspace{1pt} \right|}

\newcommand{\supp}[1]{\mathop{\mathrm{supp}}\left( #1\right)}

\newcommand{\close}[1]{\overline{#1}}

\newcommand{\mr}[1]{\mathrm{#1}}

\newcommand{\diam}{\mathop{\mathrm{diam}}}

\newcommand{\hrs}[1]{\hspace{#1pt}}
\newcommand{\vrs}[1]{\vspace{#1pt}}

\newcommand{\Arg}[1]{\mathrm{Arg}\hrs{1}{#1}} 

\newcommand{\integ}[1]{\mathbb{Z}^{\hrs{0}#1}}

\newcommand{\real}[1]{\mathbb{R}^{\hrs{0}#1}}
\newcommand{\compx}[1]{\mathbb{C}^{\hrs{0}#1}}

\newcommand{\td}[1]{\widetilde{#1}}
\newcommand{\wh}[1]{\widehat{#1}}

\newcommand{\rhull}[1]{\mathcal{R}\mathrm{-hull}\left( {#1} \right)}

\newcommand{\Chi}{\mathcal{X}}

\newcommand{\ee}{\varepsilon}

\newcommand{\dd}{\delta}

\newcommand{\pdi}[2]{\frac{\partial {#1}}{\partial {#2}}}

\renewcommand{\it}[1]{\textit{#1}}
\renewcommand{\bf}[1]{\textbf{#1}}

\title{A Characterization of Rationally Convex Immersions}

\author{Octavian Mitrea}
\address{Department of Mathematics, the University of Western Ontario, London, Ontario, N6A 5B7, Canada}
\date{}

\begin{document}
\maketitle

\begin{abstract}
Let $S$ be a smooth, totally real, compact immersion in $\compx{n}$ of real dimension $m \leq n$, which is locally polynomially convex and it has finitely many points where it self-intersects finitely many times, transversely or non-transversely. We prove that $S$ is rationally convex if and only if it is isotropic with respect to a ``degenerate" K\"ahler form in $\compx{n}$.
\end{abstract}

\let\thefootnote\relax
\let\thefootnote\relax\footnote{
The author is partially supported by the Natural Sciences and Engineering Research Council of Canada.
}

\section{Introduction}\label{sec:1}

The purpose of this paper is to prove the following characterization of a class of rationally convex, totally real immersions in $\compx{n}$ of real manifolds.

\begin{theorem} \label{th1} Let $S$ be a smooth compact manifold of dimension $m \geq 1$, with or without boundary, and let $\iota :S \rightarrow \compx{n}$, $n \geq m$, be an immersion such that $\iota(S)$ is a smooth submanifold of $\compx{n}$, except at finitely many points $p_1, p_2, \dots p_N \in \iota(S)$, where $\iota(S)$ intersects itself finitely many times. Suppose $\iota(S)$ is totally real and locally polynomially convex. Then the following are equivalent:
\begin{enumerate}[(i)]
\item $\iota(S)$ is rationally convex;
\item There exist contractible neighborhoods $W_j$ of $p_j$ in $\iota(S)$, $j=1, \dots, N$, such that for every neighborhood $\Omega$ of $\iota(S)$, there exist neighborhoods $U_j \subset V_j$ of $p_j$ in $\compx{n}$, $j=1, \dots, N$, with $\{\close{V_j}\}_j$ pairwise disjoint, and a smooth plurisubharmonic function $\varphi : \compx{n} \rightarrow \real{}$, satisfying the following properties:
\begin{enumerate}
\item $U_j \cap \iota(S) = W_j$, $j=1,\dots,N$;
\item $\cup_{j=1}^N V_j$ is compactly included in $\Omega$;
\item $dd^c \varphi =0$ on $\cup_{j=1}^N U_j$;
\item $\varphi$ is strictly plurisubharmonic on $\compx{n} \setminus \cup_{j=1}^N\close{V_j}$;
\item $\iota^* dd^c \varphi = 0$.
\end{enumerate}

\end{enumerate}
\end{theorem}

A first characterization of the rational convexity of a smooth totally real compact submanifold $S \subset \compx{2}$ was given by Duval \cite{12}, \cite{13}, who showed that $S$ is rationally convex if and only if $S$ is Lagrangian with respect to some K\"ahler form in $\compx{2}$. Subsequently, applying a different method that makes use of H\"ormander's $L^2$ estimates, Duval and Sibony \cite{1} extended the result to totally real embeddings of any dimension less than or equal to $n$. It is thanks to these remarkable results that the intrinsic connection between rational convexity and symplectic properties of real submanifolds has been revealed. In \cite{2} Gayet analyzed totally real immersions in $\compx{n}$ of maximal dimension, with finitely many transverse double self-intersection points, showing that being Lagrangian with respect to some K\"ahler form in $\compx{n}$ is a sufficient condition for such immersions to be rationally convex. A similar result was proved later by Duval and Gayet \cite{14} for immersions of maximal dimension with certain non-transverse intersections.

Theorem \ref{th1} gives a characterization of the rational convexity of a more general class of immersions in $\compx{n}$ with finitely many self-intersection points: we do not impose any restrictions on the real dimension of such immersions and do not require the self-intersections to be transverse or double. The proof of the theorem spans two sections of this paper. In Section \ref{sec:3} we prove Proposition~\ref{prp:1} which shows that the implication $(i) \Rightarrow (ii)$ in Theorem~\ref{th1} is true. It is important to note that the proof of Proposition~\ref{prp:1} does not require $S$ to be polynomially convex near the singular points. Proposition~\ref{prp:2} proved in Section \ref{sec:4} shows that the other direction of Theorem~\ref{th1} holds true. In the proof we follow closely the method introduced in \cite{1} (see also \cite{2}, \cite{6}). The condition for $S$ to be polynomially convex near the singular points plays a key role in the proof. Note, however, that all the submanifolds considered in \cite{12}, \cite{13}, \cite{1}, \cite{2} are polynomially convex near every point. This is a classic result for smooth totally real embeddings and, in the case of the class of immersions considered in \cite{2}, the property follows from a result of Shafikov and Sukhov \cite[Theorem 1.4]{10} who showed that every Lagrangian immersion with finitely many transverse self-intersections is locally polynomially convex. The second example in Section \ref{sec:5} shows that in general an immersion that is not locally polynomially convex may fail to be rationally convex. It is natural to ask whether local polynomial convexity is also guaranteed for immersions that are isotropic with respect to a "degenerate" K\"ahler form, as described in Theorem~\ref{th1}. This remains to be an open problem

In Section \ref{sec:5}, using a theorem of Weinstock \cite{4}, we show that there is a ``large" family of compact, totally real immersions in $\compx{2}$ with one transverse self-intersection, which are rationally convex but are not isotropic with respect to any K\"ahler form on $\compx{2}$, thus Gayet's theorem \cite{2} cannot be applied to this case. However, by Theorem~\ref{th1} they are isotropic with respect to a degenerate K\"ahler form.

We also remark that the main result in \cite{2} is implied by Theorem~\ref{th1}. Indeed, in the first step of the proof of the theorem in \cite{2} it is shown that $S$ being Lagrangian with respect to a K\"ahler form $\omega = dd^c \varphi$ implies that $S$ is isotropic with respect to a nondegenerate closed $(1,1)$-form defined on $\compx{n}$, that vanishes on sufficiently small neighborhoods of the singular points and is positive in the complement of some slightly larger neighborhoods. This is done by composing the original potential $\varphi$ with a suitable non-decreasing, convex function. We note that, prior to such composition, one can multiply $\varphi$ with a suitable cutoff function, and then apply the composition, this way controlling the shape of the neighborhoods mentioned above. It follows that the hypothesis of the main result in \cite{2} implies the hypothesis of the direction $(ii) \Rightarrow (i)$ in Theorem~\ref{th1}.

\bf{Acknowledgment.} I would like to thank my advisor, Rasul Shafikov, for his most valuable guidance, constant support and inspiring mentorship offered throughout the research process that led to the main result of this paper.


\section{Preliminaries} \label{sec:2}

In this section we review the basic background necessary for understanding the paper. Unless otherwise specified, by ``smooth" we shall mean $\mathcal{C}^\infty$-smooth and by a neighborhood of a compact connected subset $X \in \compx{n}$ we shall mean a connected open set containing $X$, having compact closure. Throughout the paper $B(p, r)$ denotes the open ball in $\compx{n}$ centered at $p \in \compx{n}$ and of radius $r>0$.


The \it{polynomially convex hull} of a compact subset $X \subset \compx{n}$ is defined as
\begin{equation*}
\widehat{X}:= \{z \in \compx{n} : \abs{P(z)} \leq \sup_{w \in X} \abs{P(w)}, \text{ for all holomorphic polynomials } P\},
\end{equation*}
and the \it{rationally convex hull} of $X$ as 
\begin{equation*}
\begin{split}
\rhull{X}:= &\{z \in \compx{n} : \abs{R(z)} \leq \sup_{w \in X} \abs{R(w)},\\& \text{ for all rational functions } R \text{ holomorphic on  } X\}.
\end{split}
\end{equation*}
We say that $X$ is \it{polynomially convex} if $X=\widehat{X}$ and \it{rationally convex} if $X = \rhull{X}$. It is immediate to see that if $X$ is polynomially convex then it is also rationally convex. $X$ is said to be \it{polynomially convex near $p \in X$} if for every sufficiently small $\ee>0$, the compact set $X \cap \close{B(p, \ee)}$ is polynomially convex. We say that $X$ is \it{locally polynomially convex} if $X$ is polynomially convex near all of its points. In most cases, it is difficult to verify whether a compact subset of $\compx{n}$ is rationally or polynomially convex. Thus, it is important to characterize these properties for regular subsets of $\compx{n}$, such as  embeddings and immersions, by using other means.

If $X$ is a submanifold of $\compx{n}$ and $p\in X$, we say that $X$ is \it{totally real at $p$} if the tangent space $T_pX$ does not contain any complex lines. $X$ is said to be \it{totally real} if it is totally real at every point. An immediate example of a totally real submanifold of the $n$-dimensional euclidean complex space is $\real{n} \subset \compx{n}$.

Let $\Omega$ be an open subset of $\compx{n}_{(z_1, \dots, z_n)}$, $z_j=x_j+iy_j$, $j=1, \dots, n$ and let $\varphi : \Omega \rightarrow \real{}$ be a real valued smooth function. As usual, we define
\begin{equation*}
\partial \varphi := \sum_{j=1}^n \frac{\partial \varphi}{\partial z_j} dz_j, \hrs{5} \close{\partial} \varphi := \sum_{j=1}^n \frac{\partial \varphi}{\partial \close{z}_j} d\close{z}_j.
\end{equation*}
It follows that the usual differential of $\varphi$ is given by $d \varphi=\partial \varphi + \close{\partial} \varphi$. We shall also need the $d^c$ differential operator, defined when acting on $\varphi$ as
\begin{equation*}
d^c \varphi :=  i(\close{\partial} \varphi - \partial \varphi),
\end{equation*}
or, in real coordinates,
\begin{equation} \label{eq:101}
d^c \varphi = \sum_{j=1}^n \left( \frac{\partial \varphi}{\partial x_j}dy_j -\frac{\partial \varphi}{\partial y_j}dx_j \right).
\end{equation}

The real valued smooth function $\varphi$ is \it{plursubharmonic} if the $(1,1)$-form $dd^c \varphi$ is nonnegative definite (or, using another common terminology, positive semidefinite). We say that $\varphi$ is \it{strictly plurisubharmonic} if $dd^c \varphi$ is positive definite. Recall that a \it{K\"ahler form}  on $\compx{n}$ is a nondegenerate closed form $\omega$ of bidegree $(1,1)$, which is positive definite. A smooth function $\varphi$ is called a \it{potential} for $\omega$ if $\omega = dd^c \varphi$. A real $m$-dimensional submanifold $S \subset \compx{n}$, $m \leq n$, is said to be \it{isotropic} with respect to a K\"ahler form $\omega$ if $\omega |_S =0$. If in the above case $m=n$ then we say that $S$ is \it{Lagrangian} with respect to $\omega$.

Let $F: D \subset \compx{n} \rightarrow \compx{n}$, $F = (u_1+iv_1, \dots, u_n+iv_n)$, be a smooth map, where $D$ is a domain in $\compx{n}$ and $u_j, v_j$ are smooth real-valued maps on $D$. For every $p \in D$ and $j \in \{1, \dots, n\}$ let 
\begin{equation*}
[d^c_pu_j]:= \left[-\frac{\partial u_j}{\partial y_1}\Bigg |_p  \hrs{5} \frac{\partial u_j}{\partial x_1}\Bigg |_p  \hrs{5} \dots  \hrs{5} -\frac{\partial u_j}{\partial y_n}\Bigg |_p  \hrs{5} \frac{\partial u_j}{\partial x_n}\Bigg |_p \right]
\end{equation*}
\begin{equation*}
[d^c_pv_j]:= \left[-\frac{\partial v_j}{\partial y_1}\Bigg |_p  \hrs{5} \frac{\partial v_j}{\partial x_1}\Bigg |_p  \hrs{5} \dots  \hrs{5} -\frac{\partial v_j}{\partial y_n}\Bigg |_p  \hrs{5} \frac{\partial v_j}{\partial x_n}\Bigg |_p \right]
\end{equation*}
be the $1 \times 2n$ matrices associated to the operator $d^c$ acting on $T_p\compx{n}$ and let
\begin{equation*}
[z]:=[x_1 \hrs{5} y_1  \hrs{5}\dots  \hrs{5} x_n  \hrs{5} y_n],
\end{equation*}
where $z=(z_1, \dots, z_n) \in T_p\compx{n}$, $z_j= x_j+iy_j$, $j=1, \dots, n$. Then, for all $z \in T_p\compx{n}$ we define
\begin{equation*}
d^c_p F(z) := [d^c_pF]\cdot [z]^T,
\end{equation*}
where $[d^c_pF]$ is the $2n\times 2n$ matrix with complex entries, given by
\begin{equation*}
[d^c_pF]:= \begin{bmatrix}
[d^c_p u_1]\\
[d^c_p v_1]\\
\dots\\
[d^c_p u_n]\\
[d^c_p v_n]
\end{bmatrix}.
\end{equation*}
The next technical result is required in the proof of Theorem~\ref{th1}.

\begin{lemma}\label{lem:1}
Let $D \subset \compx{n}$ be a domain, $F : D \rightarrow \compx{n}$, $F = (F_1, \dots, F_n)$, a smooth map such that $F(D)$ is a domain in $\compx{n}$ and $h : F(D) \rightarrow \real{}$ a smooth function. Then
\begin{equation*} \label{eq:1}
d^c_p(h \circ F) = d_{F(p)}h \circ d^c_pF,
\end{equation*}
at any point $p \in D$.
\end{lemma}
\begin{proof}
Let $p \in D$. By the complex chain rule (see for example \cite[p.6]{3}),
\begin{equation*}\label{eq:2}
\partial_p(h \circ F) = \partial_{F(p)}h \circ \partial_pF + \bar{\partial}_{F(p)}h \circ \bar{\partial}_pF,
\end{equation*}
\begin{equation*}\label{eq:3}
\bar{\partial}_p(h \circ F) = \bar{\partial}_{F(p)}h \circ \partial_pF + \partial_{F(p)}h \circ \bar{\partial}_pF.
\end{equation*}
So,
\begin{equation*}\label{eq:4}
\begin{split}
d^c_p(h \circ F) &= i\left[ \bar{\partial_p}(h\circ F) - \partial_p{(h \circ F)} \right]\\& =i\left[ \bar{\partial}_{F(p)}h \circ \partial_pF + \partial_{F(p)}h \circ \bar{\partial}_pF - \partial_{F(p)}h \circ \partial_pF - \bar{\partial}_{F(p)}h \circ \bar{\partial}_pF \right]\\& =  \bar{\partial}_{F(p)}h \circ[ -i(\partial_p F- \bar{\partial}_pF)]  + \partial_{F(p)}h\circ[i(\bar{\partial}_pF -\partial_pF)] \\ &= (\bar{\partial}_{F(p)}h + \partial_{F(p)}h)\circ d^c_pF\\ &=  d_{F(p)}h \circ d^c_pF.
\end{split}
\end{equation*}
\end{proof}

One other important tool we will make use of is the standard Euclidean \it{distance function} defined for a subset $M \subset \compx{n}$ as
\begin{equation*}
dist(z, M) = \inf\{dist(z, p) : p \in M\}
\end{equation*}
for all $z \in \compx{n}$. It is a well known fact that, if $M$ is a smooth totally real submanifold of $\compx{n},$ the square distance function, $dist^2(\cdot, M)$, is smooth and strictly plurisubharmonic on a sufficiently small neighborhood of $M$.

For the rest of the paper, we shall commit a mild abuse of notation and keep the notation $S$ for the image in $\compx{n}$ of the given manifold $S$ via the immersion..

\section{The Necessary Condition for the Rational Convexity of $S$} \label{sec:3}

In this section we prove that $(i)$ implies $(ii)$ in Theorem~\ref{th1}. As we already mentioned in the introduction, in this case we do not require the totally real immersion to be locally polynomially convex. In fact, we shall prove the following more general result, where the rational convexity of $S$ implies the existence of a family of of degenerate K\"ahler forms with respect to which it is isotropic.

\begin{proposition} \label{prp:1}
Let $S$ be the immersion defined in Theorem~\ref{th1} without assuming that it is locally polynomially convex. If $S$ is rationally convex then for every sufficiently small $\ee >0$ there exist contractible neighborhoods $W_\ee^j$ of $p_j$ in $S$, $j=1, \dots, N$, such that for every neighborhood $\Omega$ of $S$ there exist neighborhoods $U_\ee^j \subset V_{\ee}^j \Subset B(p_j, \ee) \cap \Omega$ of $p_j$, $j = 1, \dots, N$ and a smooth plurisubharmonic function $\varphi_{\ee} : \compx{n} \rightarrow \real{}$ such that $U_{\ee}^j \cap S = W_\ee^j$ for all $j$, $dd^c \varphi_{\ee} =0$ on $\cup_{j=1}^N \close{U_{\ee}^j}$, $\varphi_\ee$ is strictly plurisubharmonic on $\compx{n} \setminus \bigcup_{j=1}^N\close{V_{\ee}^j}$ and $\iota^* dd^c \varphi_{\ee} = 0$.
\end{proposition}

A consequence of the above proposition is the following corollary which may be useful in applications. Its proof is given at the end of this section.
 
\begin{corollary} \label{cor:1}
If $S$ is rationally convex then for all integers $k \geq 2$ there exists a $C^k$-smooth plurisubharmonic function $\varphi_0 : \compx{n} \rightarrow \real{}$ which is strictly plurisubharmonic on $\compx{n} \setminus \{p_1, \dots, p_N\}$ and such that $\iota^* dd^c \varphi_0 = 0$.
\end{corollary}

Fix $j \in \{1, \dots, N\}$ and suppose that $p_j$ is a point where $S$ self-intersects $l$ times. For a sufficiently small $r>0$, the set $S \cap B(p_j,r)$ is the union of $l$ compact smooth submanifolds with boundary, $S_1, \dots,  S_l \subset S$, such that $S_k \cap S_m = \{p_j\}$, $k \neq m$. We say that $S_1, \dots, S_l$ are \it{smooth components of $S$ at $p_j$}. The proof of Proposition~\ref{prp:1} relies on the construction of a suitable function defined on a neighborhood of each smooth component of $S$ at each singular point and on the patching of all such functions into one that has the required properties (as per Lemma~\ref{lem:2}). The following two lemmas (\ref{lem:10} and \ref{lem:2}) will be applied separately to each smooth component of $S$ at each singular point. More generally, we prove the lemmas for a smooth, totally real submanifold $M$ of $\compx{n}$, with or without boundary.

\begin{lemma} \label{lem:10}
For every point $p \in M$ there exists a smooth function $\td{f} : \compx{n} \rightarrow \real{}$, with compact support, $p\in supp (\td{f})$, such that $\td{f}$ has a local minimum at $p$ and satisfies $\iota^* d^c \td{f} = 0$, where $\iota : M \rightarrow \compx{n}$ is the inclusion map. 
\end{lemma}
\begin{remark}
Note that the lemma does not require $p$ to be a \it{strict} local minimum point for $f$.
\end{remark}
\begin{proof}[Proof of Lemma~\ref{lem:10}.]
Suppose first that $\dim_{\real{}}M = n$. For each $q \in M$, there exists a global complex-affine change of coordinates, \bf{which depends on $q$}, 
\begin{equation*}\label{41}
\Phi : \compx{n}_{ z=(z_1, \dots, z_n) }\rightarrow \compx{n}_{ w=(w_1, \dots, w_n) },
\end{equation*}
where $z_j=x_j+iy_j, w_j=u_j+iv_j$, $x_j, y_j, u_j, v_j \in \real{}, \forall j=1 \dots n$, such that $\Phi(q)=0$ and $T_0M' = \real{n}_{u := (u_1, \dots, u_n)}$, with $M'=\Phi(M)$. Suppose that $\Phi(z) = A(z-q)$, where $A$ is a complex $n\times n$ invertible matrix (as a complex-affine map, we can always represent $\Phi$ like this). Let $J :=  \begin{bmatrix}
0&-I_n\\
I_n&0
\end{bmatrix}$ be the matrix that gives the standard complex structure of $\compx{n}$, corresponding to multiplication by $i$. Denote 
\begin{equation*}
JT_qM := \{ (i(z-q)+q \hrs{2} | \hrs{2} z \in T_qM\}.
\end{equation*}

\noindent \bf{Claim A.} $\Phi(JT_qM) = \real{n}_v \subset \compx{n}_w$, for all $q \in M$.
\begin{proof}[Proof of Claim A]
Let $w \in JT_qM$, hence there exists $z \in T_qM$ such that $w = i(z-q)+q$. Then, 
\begin{equation*}\label{eq:5}
\begin{split}
\Phi(w) &= \Phi(i(z - q) + q)\\ &=A \left[i(z - q) + q -q  \right]\\ &= iA(z - q) = i\Phi(z) \in JT_0M' = \real{n}_v.
\end{split}
\end{equation*}
The converse inclusion follows similarly.
\end{proof}

Let $p \in M$ be arbitrarily fixed. Since $M$ is totally real, there exists a (small enough) neighborhood $U$ of $p$ such that 
\begin{equation*}\label{eq:6}
\mathcal{F}_U:=\{JT_qM \cap U \hrs{2}| \hrs{2} q \in M \cap U\}
\end{equation*}
is a foliation of $U$. 

Let $V$ be a neighborhood of $p$ in $M$ such that $V \Subset M\cap U$ and let $f:M \rightarrow \real{}$ be a smooth nonnegative function such that $supp \hrs{2} f = \close V$, with a strict local minimum, equal to $0$, at $p$. Define $\td{f} : U \rightarrow \real{}$ as
\begin{equation*}\label{61}
\td{f}(z) = f(q), \text{ for each } z \in U \cap JT_qM, q \in M \cap U,
\end{equation*}
which is well defined, since $\mathcal{F}_U$ is a foliation of $U$ and, shrinking $U$ if necessary, for each $q \in M \cap U$ we have $M \cap U\cap JT_qM = \{q\}$.
Multiply $\td{f}$ with a suitable smooth cut-off function to obtain a new $\td{f}$ (maintaining the same notation) which is defined on the entire $\compx{n}$ and satisfies
\begin{equation*}\label{eq:7}
\td{f}(z)=
\begin{cases}
f(q), \hrs{5} &\forall z \in JT_qM \cap U', q\in V,\\
0, &\forall z \in \compx{n}\setminus U'',
\end{cases}
\end{equation*}
where $U' \subset U'' \Subset U$ and $U' \cap M =V$.
Note that $\td{f}$ is nonnegative, smooth, with compact support and with $0$ as a non-strict local minimum value at $p$.

For every $q \in V$, define
\begin{equation*}\label{eq:8}
h:= \td{f} \circ \Phi^{-1},
\end{equation*}
where, again, $\Phi$ (and therefore $h$) depends on the choice of $q$. Then, $h$ is a smooth real-valued function satisfying
\begin{equation*}\label{eq:9}
h(w) = h(0) = f(q), \hrs{5} \forall w\in W\cap J\real{n}_u = W\cap \real{n}_{v=(v_1, \dots, v_n)},
\end{equation*}
for some neighborhood $W$ of the origin $0 \in \compx{n}_w$. Hence, $h$ is constant in $W \cap \real{n}_v$, which means that
\begin{equation}\label{eq:10}
\frac{\partial h}{\partial v_j}(0) = 0,
\end{equation}
for all $j=1 \dots n$. By (\ref{eq:101}) we have
\begin{equation*}\label{eq:11}
d^c = \sum_{j=1}^n \left( \frac{\partial}{\partial u_j}dv_j -\frac{\partial}{\partial v_j}du_j \right),
\end{equation*}
and by (\ref{eq:10}) and the fact that $dv_j =0$ on $\real{n}_u$ we get
\begin{equation}\label{eq:12}
j^* d_0^c h =0,
\end{equation}
where $j:\real{n}_u \rightarrow \compx{n}$ is the inclusion map.

\noindent \bf{Claim B.} If $j^* d_0^c h =0$ then $\iota^*d_q^c\td{f} =0$.

\begin{proof}[Proof of Claim B]
By Lemma~\ref{lem:1}, 
\begin{equation} \label{eq:14}
d_0^ch = d^c_0(\td{f}\circ\Phi^{-1}) = d_{\Phi^{-1}(0)}\td{f} \circ d^c_0 \Phi^{-1} = d_q\td{f} \circ d^c_0 \Phi^{-1}.
\end{equation}
Let $\nu$ be a tangent vector in $T_qM$. Then, $\displaystyle d^c_q\td{f}(\nu) = i\bar{\partial}_q\td{f}(\nu) - i\partial_q\td{f}(\nu) = \bar{\partial}_q\td{f}(-i\nu) + \partial_q\td{f}(-i\nu) = d_q\td{f}(-i\nu)$, where we used the complex anti-linearity of $\bar{\partial}$ and the complex linearity of $\partial$.
Because $\Phi$ is (bi)holomorphic, we have $\displaystyle d^c_0\Phi^{-1} = -id_0\Phi^{-1}$. Since $d_0\Phi^{-1}$ is a vector space isomorphism, there exists $\xi \in T_0M' = \real{n}_u$ such that $\nu = d_0\Phi^{-1}(\xi)$, hence $-i\nu = -id_0\Phi^{-1}(\xi) = d^c_0\Phi^{-1}(\xi)$. Since $\nu$ was arbitrarily fixed in $T_qM$ and by (\ref{eq:14}), the claim follows.
\end{proof}

By Claim B and by (\ref{eq:12}) it follows that $\iota^* d^c_q\td{f} =0$ and, since $q$ was arbitrarily fixed in $V$ and by the fact that on $\compx{n} \setminus U''$ we have $\td{f}\equiv 0$, we conclude that
\begin{equation}\label{eq:15}
\iota^*d^c\td{f} = 0.
\end{equation}

If $\dim_{\real{}}M < n$ and $p \in M$, there exists $U$, a neighborhood of $p$ in $\compx{n}$, such that $M \cap U$ is included in a compact, totally real submanifold $\td{M}$ of $\compx{n}$ of real dimension $n$. By what we proved already, there exists a smooth function $\td{f}:\td{M} \rightarrow \real{}$, with compact support, which can be chosen such that $supp (f) \subset U$ and with a local minimum at $p$, such that $\td{\iota}^*d^c\td{f} = 0$, where $\td{\iota}:\td{M} \rightarrow \compx{n}$ is the inclusion map. Then, the restriction $\td{f} \big|_M$ satisfies the same properties for $M$. This completes the proof of Lemma~\ref{lem:10}.
\end{proof}

\begin{lemma}\label{lem:2}
For every point $p \in M$ and every sufficiently small $\ee>0$, there exist a neighborhood $\Omega_\ee$ of $M$, a smooth nonnegative function $g_\ee : \Omega_\ee \rightarrow \real{}$ and neighborhoods $V_\ee , W_p \subset \Omega_\ee$ of $p$ such that
\begin{enumerate}[$(i)$]
\item $V_\ee \Subset B(p,\ee) \Subset W_p$;
\item $g_\ee\equiv 0$ in $V_\ee$;
\item $g_\ee$ is plurisubharmonic in $\Omega_\ee$ and strictly plurisubharmonic in $\Omega_\ee \setminus \close{V}_\ee$;
\item $g_\ee = C_\ee \cdot dist^2(\cdot, M)$ in $\Omega_\ee \setminus W_p$, for some constant $C_\ee>0$;
\item $\iota^*d^c g_\ee = 0$.
\end{enumerate}
\begin{remark}
The neighborhood $W_p$ depends only on $p$, not on $\ee$, therefore the notation.
\end{remark}
\end{lemma}
\begin{proof}[Proof of Lemma~\ref{lem:2}]

Without loss of generality, suppose $p=0 \in \compx{n}$. By Lemma~\ref{lem:10}, there exists a smooth function $\td{f}$ defined on $\compx{n}$, such that $\td{f}$ has compact support near the origin,  $\iota^*d^c \td{f} = 0$ and $\td{f}$ has a local minimum at the origin. In fact, by construction, $\td{f}$ is nonnegative, with a (non-strict) local minimum equal to $0$ attained on $JT_0M$. Let $W_p:=(\supp{\td{f}})^\circ$.

Suppose $\td{\Omega}$ is a neighborhood of $M$ on which $dist^2(\cdot, M) \big |_{\td{\Omega}}$ is smooth and strictly plurisubharmonic. We can also assume that $W_p \Subset \td{\Omega}$.  Then, for a sufficiently small $C>0$, the function defined on $\td{\Omega}$ as
\begin{equation}\label{eq:16}
\td{g}= dist^2(\cdot, M) + C\td{f},
\end{equation}
satisfies,
\begin{enumerate}[$(a)$]
\item $\td{g}$ is strictly plurisubharmonic in $\td{\Omega}$;
\item $\td{g}$ is nonnegative and it has a strict local minimum at the origin which is equal to $0$;
\item $\td{g} = dist^2(\cdot, M)$ on $\td{\Omega} \setminus W_p$;
\item $\iota^*d^c\td{g}=0$, by Lemma~\ref{lem:10} and by the fact that $\iota^* d^c \left[dist^2(\cdot, M)\right] =0$.
\end{enumerate}

Let $0<\ee'<\ee$ be sufficiently small, by which we mean that $B(0, \ee)$ is included in a neighborhood of the origin $U \Subset W_p$ on which $\td{g}$ has a strict minimum at the origin, and let $a_\ee := \max\{\td{g}(z) : z\in \close{B(0, \ee')} \}$. By making $\ee'$ even smaller if necessary, there exists a neighborhood $0\in V_\ee \subset B(0, \ee)$ such that $\td{g}(z)>a_\ee$ for all $z \in U \setminus \close{V}_\ee$ and $\td{g}(z)\leq a_\ee$ for all $z \in \close{V}_\ee$. 

Define $\sigma_\ee : \real{}_{\geq 0} \rightarrow \real{}$ to be a nonnegative smooth, convex, non-decreasing function such that $\sigma_\ee(t)=0$ for all $t \in  [0, a_\ee]$. It follows that the function 
\begin{equation*}\label{eq:22}
\td{g}_\ee := \sigma_\ee \circ \td{g}
\end{equation*}
is identically zero in $\close{V}_\ee$. Since $dist^2(\cdot, M) = 0$ on $M$ and $f=0$ on the complement of $W_p =(\supp{f})^\circ$, the function  $\td{g}$ vanishes on $M \setminus W_p$, hence $\td{g}_\ee$ vanishes on a neighborhood $\Omega'_\ee$ of $M \setminus W_p$. Note that, for sufficiently small $\ee>0$ (hence, for sufficiently small $a_\ee$), we can ensure that $\Omega'_\ee \cap V_\ee = \emptyset$. Since $\td{g}$ is strictly plurisubharmonic in $\td{\Omega}$, $\td{g}_\ee$ is plurisubharmonic in $\td{\Omega}$ and strictly plurisubharmonic in $\td{\Omega} \setminus (\close{\Omega'}_\ee \cup \close{V}_\ee)$. Moreover, since $\iota^*d^c\td{g} = 0$ and
\begin{equation*}\label{eq:23}
dd^c\td{g}_\ee = dd^c(\sigma_\ee \circ \td{g}) = \sigma_\ee''d\td{g} \wedge d^c\td{g} + \sigma_\ee'dd^c\td{g},
\end{equation*}
it follows that $\iota^*dd^c\td{g}_\ee = 0$.

The set $\Omega_\ee:=\Omega'_\ee\cup W_p$ is a neighborhood of $M$. Let $\Chi : \compx{n} \rightarrow \real{}$ be a smooth cut-off function, which is identically $0$ on a neighborhood of the origin $Z$ and equal to $1$ on the complement of a slighter larger neighborhood $Z'$, where $V_\ee \subset Z \Subset Z' \Subset W_p$. We can also ensure that $Z' \cap \Omega'_\ee = \emptyset$ (because $\Omega'_\ee \cap V_\ee = \emptyset$). For a sufficiently small constant $C_\ee>0$, the function defined on $\Omega_\ee$ as
\begin{equation*}\label{eq:221}
g_\ee:= \td{g}_\ee + C_\ee(\Chi \cdot dist^2(\cdot, M))
\end{equation*}
has the required properties.
\end{proof}

Let $S$ be the immersion considered at the beginning of this section, with finitely many double self-intersections, $p_1, \dots, p_N \in S$.

\begin{lemma}\label{lem:3}
For any sufficiently small $\ee>0$ there exist a neighborhood $\Omega_\ee$ of $S$, neighborhoods $p_j \in V_\ee^j \subset \Omega_\ee$, with $\diam V_\ee^j < \ee$ for all $j=1\dots N$, and a smooth function $\rho_\ee : \Omega_\ee \rightarrow \real{}$ such that
\begin{enumerate}[$(i)$]
\item $\rho_\ee \equiv 0$ in $V_\ee^j$, $j=1\dots N$;
\item $\rho_\ee$ is plurisubharmonic in $\Omega_\ee$ and strictly plurisubharmonic in $\Omega_\ee \setminus \cup_{j=1}^N \close{V_\ee^j}$;
\item $\iota^*dd^c\rho_\ee = 0$.
\end{enumerate}
\end{lemma}
\begin{proof}
Without any loss of generality we can suppose that $S$ has only one (double) self-intersection at the origin. The construction can be easily extended to the general case.

Let  $S_1,S_2$ be two smooth components of $S$ at the origin and let $\ee>0$ be sufficiently small. By Lemma~\ref{lem:2}, it follows that, for $j=1,2$, there exist $\Omega_\ee^j$, a neighborhood of $S_j$, a smooth nonnegative function $g_\ee^j : \Omega_\ee^j \rightarrow \real{}$ and neighborhoods $0\in V_\ee^j \Subset B(0, \ee) \Subset W^j \subset \Omega_\ee^j$ such that 
\begin{enumerate}[$(i)$]
\item $g_\ee^j$ is plurisubharmonic in $\Omega_\ee^j$ and strictly plurisubharmonic in $\Omega_\ee^j \setminus \close{V}_\ee^j$;
\item $g_\ee^j = 0$ in $V_\ee^j$;
\item $g_\ee^j = C_\ee \cdot dist^2(\cdot, S_j)$ in $\Omega_\ee^j \setminus W^j$, where $C_\ee := \min\{C_\ee^1, C_\ee^2\}$ and the constants $C_\ee^j$ are given by Lemma~\ref{lem:2};
\item $\iota^*dd^cg_\ee^j = 0$.
\end{enumerate}

Let $V_\ee := V_\ee^1 \cap V_\ee^2$. By construction, $V_\ee \Subset B(0, \ee) \Subset \Omega_\ee^j, j=1,2$. Make the neighborhoods $\Omega_\ee^1, \Omega_\ee^2$ narrow enough  so that $(\Omega_\ee^1 \cap \Omega_\ee^2) \setminus V_\ee = \emptyset$. Let $\Omega_\ee := \Omega_\ee^1 \cup \Omega_\ee^2 \cup V_\ee$, and define $\rho_\ee : \Omega_\ee \rightarrow \real{}$ as
\begin{equation*}\label{eq:24}
\rho_\ee(z) =
\begin{cases}
0, &z \in  \close{V_\ee},\\
g_\ee^1(z), &z \in \Omega_\ee^1 \setminus \close{V_\ee},\\
g_\ee^2(z), &z \in \Omega_\ee^2 \setminus \close{V_\ee},
\end{cases}
\end{equation*}
which satisfies
\begin{equation*}\label{eq:25}
\rho_\ee(z) =
\begin{cases}
0, &z \in  \close{V_\ee},\\
C_\ee \cdot dist^2(\cdot, S), &z \in \Omega_\ee \setminus W_p,
\end{cases}
\end{equation*}
where $W_p = W_p^1 \cup W_p^2$. Lastly, extend $\Omega_\ee$ to a full neighborhood of $S$ and $\rho_\ee$ accordingly to obtain the required function.
\end{proof}

In the following, we set $S_\delta := \{z \in \compx{n} : dist(z, S) \leq \dd\}$, where $\delta >0$.
\begin{lemma} \label{lem:40}
For every neighborhood $\Omega$ of $S$, there exists $\delta_0 >0$, which depends on $\Omega$, such that $\rhull{S_\delta} \Subset \Omega$ for all $0<\delta \leq \delta_0$.
\end{lemma}
\begin{proof}
Let $\Omega$ be a neighborhood of $S$. It suffices to show that there exists $\nu_0 \in \integ{+}$ such that $\rhull{S_{1/\nu}} \Subset \Omega$ for all integers $\nu \geq \nu_0$. In fact, because $S_{1/{(\nu+1})} \subset S_{1/\nu}$ for all $\nu$, it is enough to prove that there exists $\nu_0 \in \integ{+}$ such that $\rhull{S_{1/\nu_0}} \Subset \Omega$. Assuming the contrary, we obtain a sequence $\{z_\nu \in \compx{n} : \nu \in \integ{+}\}$ such that $z_\nu \in \rhull{S_{1/\nu}}$ and $z_\nu \not \in \Omega$ for all $\nu \in \integ{}$. Since $\rhull{S_1}$ is compact we may assume that this sequence converges to some $z \in \compx{n}\setminus \Omega$, which means that $z \not \in S$, since $S \subset \Omega$. Recall that for any compact $X \subset \compx{n}$ we have (see for example \cite[Proposition 1.1]{7})
\begin{equation*} \label{eq:27}
\rhull{X} = \{z \in \compx{n} : f(z) \in f(X), \text{ for all holomorphic polynomials } f\}.
\end{equation*}
Since $z \not \in S$ and $S$ is rationally convex, there exists a holomorphic polynomial $P$ such that $P(z) \not \in P(S)$. By a continuity/compactness type of argument one can easily show that $P(S)=\cap_{\nu=1}^\infty P(S_{1/\nu})$. Since $P(z) \not \in P(S)$ and $P(S)$ is compact, there exists $r>0$ such that $B(P(z), r) \cap P(S) = \emptyset$ so, for all but finitely many elements of the sequence, we have $P(z_\nu) \in B(P(z), r)$. On the other hand, since $P(S)=\cap_{\nu=1}^\infty P(S_{1/\nu})$, there exists a neighborhood $U$ of $P(S)$ such that $U \cap B(P(z), r) = \emptyset$ and all but finitely many elements $P(z_\nu)$ belong to $U$, which leads to a contradiction.
\end{proof}

\begin{proof}[Proof of Proposition~\ref{prp:1}]
For $\ee >0$ sufficiently small, let $\Omega_\ee$ be the neighborhood of $S$, $\td{V}_\ee^j \subset \Omega_\ee$, $\diam \td{V}_\ee^j < \ee$, the neighborhoods of the self-intersection points $p_j$, $j=1, \dots, N$, and $\rho_\ee$ the function given by Lemma~\ref{lem:3}. Let $\Omega$ be a neighborhood of $S$ such that $\Omega \Subset \Omega_\ee$. By Lemma~\ref{lem:40} there exists $\delta >0$ such that $\rhull{S_\delta} \Subset \Omega$. By Duval-Sibony \cite[Theorem 2.1]{1} there exists a smooth plurisubharmonic function $\psi_\delta : \compx{n} \rightarrow \real{}$ which is strictly plurisubharmonic on $\compx{n} \setminus \rhull{S_\delta}$ and satisfies $dd^c\psi_\delta|_{\rhull{S_\delta}} = 0$. Let $W_\ee^j := \td{V}_\ee^j \cap S$, which for sufficiently small $\ee>0$ is contractible, $U_\ee^j:= \td{V}_\ee^j \cap S_\dd$, $V_\ee^j := \td{V}_\ee^j \cap \Omega$ and define 
\begin{equation*}\label{eq:26}
\varphi_\ee(z) := \psi_\delta + C\Chi(z) \rho_\ee(z),
\end{equation*}
where $C$ is a positive constant and $\Chi$ is a smooth cutoff function equal to $1$ in a neighborhood of $S$ which contains $\Omega$ and equal to $0$ in the complement of a slightly larger neighborhood  of $S$, both neighborhoods being compactly included in $\Omega_\ee$. Then, for a sufficiently small $C>0$, $\varphi_\ee$ is strictly plurisubharmonic in $\compx{n}\setminus \bigcup_{j=1}^k\close{V_\ee^j}$ and the rest of the properties stated in Proposition~\ref{prp:1} are satisfied. 
\end{proof}

\begin{proof}[Proof of Corollary \ref{cor:1}]
Let $\varphi_j := \varphi_{\ee_j} : \compx{n}_{z=(z_1, \dots, z_n)} \rightarrow \real{n}$, $z_\mu = x_\mu+ix_{\mu+n}$, $\mu = 1, \dots, n$, where $\displaystyle\{\ee_j\}_{j\in \integ{+}}$ is a decreasing sequence of (sufficiently small) positive numbers converging to $0$ and $\varphi_{\ee_j}$ are the functions given by Proposition~\ref{prp:1}. Let $B \subset \compx{n}$ be a closed ball such that $\cup_j \Omega_{\ee_j} \Subset B$, where $\Omega_{\ee_j}$ are the neighborhoods of $S$ given by Lemma~\ref{lem:3}. Then, there exist positive reals, $\alpha_j >0$, $j \in \integ{+}$, such that 
\begin{equation*}
\varphi_0 := \sum_j \alpha_j\varphi_j <\infty, \hrs{10} \sum_j \alpha_j \frac{\partial^l \varphi_j}{\partial x_{\mu_1} \dots \partial x_{\mu_l}} < \infty
\end{equation*}
in $B$, for all $1 \leq l \leq k$ and $\mu_1, \dots, \mu_l \in \{ 1, \dots, 2n\}$ where the convergence of the series is uniform. By making use of a classic result in real single variable calculus, see for example \cite[ Theorem 7.17]{11}, it is straightforward to show that $\varphi_0$ is $C^k$-smooth in $B$ and that
\begin{equation*}
\frac{\partial^l \varphi_0}{\partial x_{\mu_1} \dots \partial x_{\mu_l}} = \sum_j \alpha_j \frac{\partial^l \varphi_j}{\partial x_{\mu_1} \dots \partial x_{\mu_l}}
\end{equation*}
for all $1 \leq l \leq k$ and $\mu_1, \dots, \mu_l \in \{ 1, \dots, 2n\}$. In particular, we have
\begin{equation} \label{eq:291}
dd^c\varphi_0 = \sum_j \alpha_j dd^c \varphi_j.
\end{equation}
In $\compx{n} \setminus B$, we have $\varphi_0 = C\psi$, where $\displaystyle C= \sum_j \alpha_j < \infty$, hence $\varphi_0$ is $C^\infty$-smooth there and clearly satisfies (\ref{eq:291}). It follows that $\varphi_0$ satisfies the required properties.
\end{proof}

\section{The Sufficient Condition for $S$ to be Rationally Convex}\label{sec:4}

Let $S$ be the immersion defined in Theorem~\ref{th1}. In this section we prove the converse statement $(ii) \Rightarrow (i)$ in Theorem~\ref{th1} which translates into the following result.

\begin{proposition} \label{prp:2}
Suppose that $S$ is locally polynomially convex and that there exist contractible neighborhoods $W_j$ of $p_j$ in $S$ such that for any neighborhood $\Omega$ of $S$, there exist a smooth plurisubharmonic function $\varphi : \compx{n} \rightarrow \real{}$ and neighborhoods $U_j \subset V_j$ of $p_j$, $j = 1, \dots, N$, where $\{V_j\}_j$ are pairwise disjoint, such that $U_j \cap S = W_j$, $\cup_{j=1}^N V_j \Subset \Omega$, $dd^c \varphi=0$ on $\cup_{j=1}^N \close{U_j}$, $\varphi$ is strictly plurisubharmonic on $\compx{n} \setminus \bigcup_{j=1}^N\close{V}_j$ and $\iota^* dd^c \varphi = 0$. Then $S$ is rationally convex.
\end{proposition}

In the proof we make use of the following technical lemma proved in \cite{2} (see also \cite{6}). In \cite{2}, the lemma is proved for totally real immersions of maximal dimension. We remark that the lemma is true in the more general case that we consider in this paper, the proof being practically the same.
\begin{lemma}\label{lem:4}
Let $\phi$ be a smooth plurisubharmonic function on $\compx{n}$ and $h$ a smooth complex-valued function on $\compx{n}$ satisfying the following properties
\begin{enumerate}[(1)]
\item $\abs{h} \leq e^\varphi$ and $X:= \{\abs{h} = e^\varphi\}$ is compact;
\item $\close{\partial}h = O(dist(\cdot, S)^{\frac{3n+5}{2}})$;
\item $\abs{h} = e^\varphi$ with order at least $1$ on $S$;
\item For any point $p \in X$ at least one of the two following conditions holds: $(i)$ $h$ is holomorphic in a neighborhood of $p$, or $(ii)$ $p$ is a smooth point of $S$ and $\varphi$ is strictly plurisubharmonic at $p$.
\end{enumerate}
Then $X$ is rationally convex.
\end{lemma}

\begin{remark} \label{rem:1}
As it was already mentioned in \cite{6}, it can be seen from the proof of the lemma in \cite{2} that it is enough to assume that $\varphi$ is only continuous and not necessarily smooth at points where $h$ is holomorphic. In fact, the Lemma is also valid if $\varphi$ is only continuous at points in the complement of a neighborhood of $S$.
\end{remark}

To prove Proposition~\ref{prp:2} we will follow closely the method used in \cite{1} (see also \cite{2}, \cite{6}). In Step~ $1$ below we construct the function $h$ from the given plurisubharmonic function, $\varphi$. The resulting pair of functions, $(h, \varphi)$, does not entirely comply with the requirements of Lemma~\ref{lem:4}, because some of the points of $S$ do not satisfy condition $(4)$ of the lemma. That is why, in the second and last step of the proof, we further perturb $\varphi$ such that the modified function is strictly plurisubharmonic at all points of $S$ where $h$ is not holomorphic. In this last step we take full advantage of the polynomial convexity of $S$ near the singular points.

\vrs{5}

\noindent \it{Step 1: construction of the function h.} The conditions $(2), (3)$ in Lemma~\ref{lem:4} that $h$ and $\varphi$ must satisfy translate into the condition $\iota^*(d^c\varphi - d(\arg h))=0 $, or  in other words, the closed $1$-form $\iota^* d^c\varphi$ has to have integer periods (see \cite{2}, \cite{6}). We can meet this condition by perturbing $\varphi$ accordingly. 

Let $\Sigma$ be an open subset of $\compx{n}$ such that $S$ is a deformation retract of $\Sigma$. Consequently, $H_1(\Sigma, \integ{}) \cong H_1(S, \integ{})$. Let $\gamma_1, \dots, \gamma_l$ be a basis for $H_1(\Sigma, \integ{})$ which we may consider to be supported on $S\setminus \cup_{j=1}^N \close{V}_j$. By the de Rham theorem, there exist closed $1$-forms $\beta_1, \dots, \beta_l$ on $\Sigma$, with compact support, such that $\int_{\gamma_j}\beta_k = \delta_{jk}$. In fact, we can choose $\beta_k$ such that they vanish on each $U_j$, $j=1, \dots, N$.

We show next that there exist smooth functions $\psi_k, k=1, \dots, l$, with compact support in $\compx{n}$, such that $\iota^* d^c\psi_k = \iota^*\beta_k$ and $\psi_k \big |_{{U}_j} \equiv 0$, $j=1, \dots, N$. Suppose first that $\dim_{\real{}}S =n$ and let $\td{S} := S \setminus \cup_{j=1}^N U_j$, which is  a compact smooth submanifold of $\compx{n}$ with boundary. Fix $k \in \{1, \dots l\}$ and, to simplify the notations, denote $\psi := \psi_k, \beta:=\beta_k$. Let $p \in \td{S}$ be arbitrarily fixed. In fact, without losing any generality, suppose $p=0$ and $T_0\td{S} = \real{n}$. Then there exists a neighborhood of the origin $W$ in $\compx{n}$ and smooth real-valued functions $r_1, \dots, r_n$ defined on $\td{W}:=W \cap \td{S}$, such that $y_j = r_j(\mr{x})$ for all $z= \mr{x}+i\mr{y} \in \td{W}$, $\mr{x} = (x_1, \dots, x_2), \mr{y}=(y_1, \dots, y_n)$, and $\pdi{r_j}{x_k}(0)=0$ for all $j, k \in  \{1, \dots, n\}$. We would like to find smooth functions $\alpha_j : W \rightarrow \real{}$, $j=1,\dots,n$, so that if we define
\begin{equation} \label{eq:281}
\psi_0(z) := \sum_{j=1}^n \alpha_j(z)\left[r_j(\mr{x}) - y_j\right],
\end{equation}
then $\iota^* d^c\psi_0 = \iota^*\beta$ on some neighborhood of the origin included in $\td{W}$. A direct calculation gives
\begin{equation*}
\iota^*d^c\psi_0 = \sum_{j=1}^n \left( \td{\alpha}_j + \sum_{k=1}^n \td{\alpha}_k P_{jk} \right) dx_j,
\end{equation*}
where $\td{\alpha}_j$ is the pullback of $\alpha_j$ on $\td{W}$, given by $\td{\alpha}_j(\mr{x})=\alpha_j(x_1, r_1(\mr{x}), \dots, x_n, r_n(\mr{x}))$, and $\displaystyle P_{jk} =  \sum_{l=1}^n \pdi{r_k}{x_j} \pdi{r_j}{x_l}$, $j,k =1, \dots, n$. For every $z \in \td{W}$ let $A(z)=[a_{jk}(z)]_{1\leq j,k \leq n}$ be the $n\times n$ matrix whose components are given by
\begin{equation*}
a_{jj}(z) = 1+P_{jj}(z),
\end{equation*}
\begin{equation*}
a_{jk}(z) = P_{jk}(z), j \neq k.
\end{equation*}
We note that $A(0)=I_n$, hence $A$ is invertible in a neighborhood $\td{Z} \subset \td{W}$ of the origin. Applying the same construction to each point $p \in \td{S}$, we obtain neighborhoods $\td{Z}_p$ in $\td{S}$ on which the corresponding matrices are non-singular. By the compactness of $\td{S}$, assume that the cover given by such neighborhoods is finite, $\{\td{Z}_\nu\}_{1\leq \nu \leq s}, s \in \integ{+}$ and denote by $A_\nu$ the corresponding matrices. Let $\{\rho_\nu\}_{1\leq \nu \leq s}$ be a smooth partition of unity subordinate to this cover. In each $\td{Z}_\nu$, write the pullback of $\beta$ in local coordinates, $\iota^*\beta = \sum_{j=1}^n \beta_jdx_j$ and let $\beta_\nu^j := \rho_\nu \beta_j$, so
\begin{equation*}
\iota^*\beta = \sum_{j=1}^n \left( \sum_{\nu=1}^s \beta_\nu^j \right) dx_j.
\end{equation*}
Then for all $1 \leq \mu \leq s$, the 
linear system in $\td{\alpha}_1^\mu(z), \dots, \td{\alpha}_n^\mu(z)$, $z \in \td{Z}_\mu$,
\begin{equation}\label{eq:280}
\td{\alpha}_j^\mu(z) + \sum_{k=1}^n \td{\alpha}_k^\mu(z) P_{jk}^\mu(z) = \sum_{\nu=1}^s \beta_\nu^j(z), \hrs{10} j=1, \dots, n.
\end{equation}
has smooth solutions in $\td{Z}_\mu$, since the system matrix $A_\mu$ is invertible there. By defining $\alpha_j^\mu (z) := \td{\alpha}_j^\mu(\mr{x})$ for all $z = \mr{x} + i\mr{y}$ in a neighborhood $Z_\mu$ in $\compx{n}$ such that $\td{Z}_\mu = Z_\mu \cap \td{S}, Z_\mu \subset W_\mu$, and by applying formula (\ref{eq:281}) accordingly, we obtain a smooth function $\psi_\mu$ defined on $Z_\mu$ which satisfies $\iota^* d^c \psi_\mu = \iota^* \beta$. Let $\Omega := \cup_{\mu =1}^s Z_\mu$, hence $\td{S} \subset \Omega$. Since $\beta$ has compact support,  $\psi_\mu$ also has compact support, so we can extend it smoothly to $\compx{n}$ by letting it be equal to zero everywhere else. Then the smooth function given by $\td{\psi} := \sum_{\mu=1}^s \psi_\mu$ satisfies the required properties. If $\dim_\real{} S < n $, we embed $\td{S}$ in a compact, totally real submanifold $\td{M} \subset \compx{n}$ of real dimension $n$, apply the above construction to obtain $\psi$ for $\td{M} \cup (\cup_{j=1}^n U_j)$ and then restrict $\psi$ to $\td{S} \cup (\cup_{j=1}^n U_j)$.

Let $\lambda:=(\lambda_1, \dots, \lambda_l)$ be an $l$-tuple of rational numbers and define $\varphi_\lambda := \varphi + \lambda_1\psi_1 + \dots +\lambda_l\psi_l$. For sufficiently small $\lambda$ (i.e., each $\lambda_k$ is sufficiently small) $\varphi_\lambda$ is strictly plurisubharmonic in $\compx{n} \setminus \cup_{j=1}^N \close{V}_j$ and pluriharmonic on each $U_j$. Further, we can find $M \in \integ{}$ by adjusting $\lambda$ accordingly, such that $\int_{\gamma_k} \iota^*d^c\varphi_\lambda \in 2\pi\integ{}/M$, $k=1, \dots, l$. Let us recycle the notation of the initial function and define $\varphi := M\varphi_\lambda$. The form $\iota^* d^c\varphi$ is closed and it has periods which are integer multiples of $2\pi$.

Fix $p \in S$ and define $\mu : S \cup(\cup_{j=1}^N \close{U}_j) \rightarrow \real{}/2\pi \integ{}$ given by $\mu(z) := \int_p^z d^c\varphi$ which clearly satisfies $\iota^*d^c\varphi = d\mu$. Fix $j \in \{1, \dots, N\}$ and $z \in U_j$. Let $\gamma : [0,1] \rightarrow \compx{n}$ be a rectifiable path such that $\gamma(0)=p, \gamma(1)=z$. Then, the function $\mu_z(w):=\int_\gamma d^c\varphi + \int_{z}^w d^c\varphi$, $w \in U_j$, does not depend on the path from $z$ to $w$, since $d^c\varphi$ is closed in $U_j$. An easy computation shows that, in a neighborhood of $z$, $\varphi$ and $\mu_z$ satisfy the Cauchy-Riemann equations, hence the function $\varphi + i\mu_z$ is holomorphic in $U_j$. Because $\mu : S \cup(\cup_{j=1}^N \close{U}_j) \rightarrow \real{}/2\pi \integ{}$ and $\Arg e^{\varphi + i\mu} = \mu$, the function $h_1 : S \cup(\cup_{j=1}^N U_j) \rightarrow \compx{n}$ given as $h_1 = e^{\varphi + i\mu}$ is well defined and holomorphic in $U_j$, $j = 1, \dots, N$. By a classic result in \cite{8}, there exists a smooth function $h : \compx{n} \rightarrow \compx{}$ such that
\begin{equation*}\label{eq:28}
h|_S = h_1|_S,
\end{equation*}
\begin{equation*}
\close{\partial} h = O(dist(\cdot, S)^r), \text{ for all } r\in \integ{+}
\end{equation*}
and $h$ is holomorphic in neighborhoods $Z_j$ of $p_j$, $Z_j \Subset U_j$, where the boundary of $Z_j$ can be chosen to be arbitrarily "close" to that of $U_j$, i.e., for any given small tubular neighborhood $V$ of $\partial U_j$, we can find such $Z_j$ with $\partial Z_j$ being included in $V$. This completes Step 1 of the proof.

\vrs{5}
\noindent \it{Step 2: further perturbation of $\varphi$.} For each singular point $p_j \in S$ there exists a region, $S \cap (\close{V_j}\setminus Z_j)$, on which we cannot guarantee $\varphi$ to be strictly plurisubharmonic, nor do we know whether $h$ is holomorphic, therefore condition $(4)$ of Lemma~\ref{lem:4} may not be satisfied there. We address this issue by taking advantage of the local polynomial convexity of $S$. We make use of the following result (Lemma~\ref{lem:5}), which is essentially due to Fostneric and Stout \cite{9} who stated it for totally real discs with finitely many hyperbolic points in $\compx{2}$. Subsequently Shafikov and Sukhov stated it for Lagrangian inclusions in $\compx{2}$ \cite[Lemma 3.3]{6} and for a different setting in $\compx{n}$ \cite[Lemma 5.3]{10}. Again, $S$ denotes the immersion defined above with $p_1, \dots, p_N$ being its self-intersection points. 

\begin{lemma} \label{lem:5}
If $S$ is locally polynomially convex then, for all sufficiently small $\delta>0$, there exists a neighborhood $\Omega$ of $S$ in $\compx{n}$ and a continuous non-negative plurisubharmonic function $\rho : \Omega \rightarrow \real{}$ such that $S=\{\rho=0\}$ and $\rho=dist(\cdot, S)^2$ on $\Omega \setminus \cup_{j=1}^N B(p_j, \delta)$; in particular, $\rho$ is smooth and strictly plurisubharmonic there.
\end{lemma}

For a proof of the lemma we refer the reader to \cite{10}. The proof is essentially identical and it applies verbatum without requiring that the self-intersections be transverse or double. In fact, the main ingredient allowing the proof to flow through is the local polynomial convexity and not the type of singularities of $S$.

To finalize Step 2 in the proof of Proposition~\ref{prp:2}, let $\delta >0$ be small enough such that $B(p_j, \delta)\cap S \Subset W_j=U_j \cap S$, for all $j$. Let $\Omega$ be the neighborhood of $S$ and $\rho : \Omega \rightarrow \real{}$ the function given by Lemma~\ref{lem:5} corresponding to $\delta$. By hypothesis, we can assume that $V_j \Subset \Omega$ for all $j \in \{1, \dots, N\}$. By the  construction in Step 1, the neighborhoods $Z_j \Subset U_j$ can be chosen so that $\partial Z_j$ is arbitrarily close to $\partial U_j$, as defined above, thus we may assume that $B(p_j, \delta) \cap S \Subset Z_j \cap S$. We extend $\rho$ smoothly on $\compx{n} \setminus \Omega$ to a function which is still denoted by $\rho$, by multiplying it by a suitable cut-off function. Then, the function
\begin{equation*}
\td{\varphi} := \varphi + C\cdot \rho
\end{equation*} 
is plurisubharmonic for a sufficiently small $C>0$. Also, since $C>0$ and $\rho$ is strictly plurisubharmonic on $\Omega \setminus \cup_{j=1}^N B(p_j, \delta)$, it follows that $\td{\varphi}$ is strictly plurisubharmonic on $S$, everywhere $h$ is not holomorphic. Removing the tilde and denoting $\td{\varphi}$ by $\varphi$, it follows that the pair $(h, \varphi)$ satisfies the conditions of Lemma~\ref{lem:4} which shows that the set $X = \{\abs{h} = e^{\varphi} \}$ is rationally convex. Clearly, $S \subset X$ and, by multiplying $h$ with a suitable cut-off function, we can construct for any neighborhood $\Omega$ of $S$ such a set $X$ which is included in $\Omega$. This proves that $S$ is rationally convex.

\begin{remark}
In \cite{2}, the final perturbation of $\varphi$ is done without the requirement of $S$ to be locally polynomially convex. We were not able to adapt this technique to our more general case.
\end{remark}

\section{Examples} \label{sec:5}

\subsection{Rationally Convex Immersions in $\compx{2}$ that are not Isotropic with respect to any K\"ahler Form }
By Theorem~\ref{th1}, $(i) \Rightarrow (ii)$, the rational convexity of $S$ implies the existence of a nonnegative closed form of bidegree $(1,1)$, $\omega := dd^c \varphi$, defined on $\compx{n}$, which is positive outside some pairwise disjoint neighborhoods $V_j$ of $p_j$, $j=1,\dots, N$, it is identically zero on some smaller neighborhoods $p_j \in U_j \subset V_j$ and $S$ is isotropic with respect to $\omega$. It is natural to ask whether the same assumption of rational convexity for $S$ leads to the existence of a genuine K\"ahler form with respect to which $S$ is isotropic, similar to the result in \cite{1}. We show in the following that there exist compact immersions in $\compx{2}$ with one transverse self-intersection which are rationally convex but which are not Lagrangian with respect to any K\"ahler form and, in fact, that this class of immersions is not just an isolated case.

Denote by $\mathcal{W}$ the set of all $2\times 2$ matrices with real entries such that for all $A \in \mathcal{W}$ the following properties are satisfied:
\begin{enumerate}[(a)]
\item $(A+i)$ is invertible;
\item $i$ is not an eigenvalue of $A$;
\item $A$ has no purely imaginary eigenvalue of modulus greater than or equal to $1$.
\end{enumerate}
It is straightforward to show that $\mathcal{W}$ is an open subset of the space of $2\times 2$ matrices with real entries, $\mathcal{M}_{2 \times 2}(\real{})$. Let $A = \begin{bmatrix}
x&y\\
z&w
\end{bmatrix} \in \mathcal{W}$ and define the following $2$-dimensional subspace of $\compx{2}$
\begin{equation*}\label{eq:29}
M(A) := (A+i)\real{2}.
\end{equation*}
It can easily be verified that condition (b) above is equivalent to $M(A)$ being totally real. By a result of Weinstock \cite{4} it follows that every compact subset of $M(A) \cup \real{2}$ is polynomially convex, hence rationally convex. For some $r>0$ let $S_1(A):= \real{2} \cap \close{B(0, r)}$, $S_2(A):= M(A) \cap \close{B(0, r)}$ and
\begin{equation*}\label{eq:30}
S(A):= S_1(A) \cup S_2(A).
\end{equation*}
Then $S(A)$ is a totally real, compact, rationally convex surface in $\compx{2}$, smooth everywhere except at the origin where it has a double self-intersection. Let $\iota_1, \iota_2 : \real{2}_{(t,s)}\rightarrow \compx{2}_{(z_1, z_2)}$ be the maps given as
\begin{gather*}\label{eq:31}
\iota_1(t,s) = (t,s),\\
\iota_2(t,s) = (xt + yt + it, zt+ws + is),
\end{gather*}
which satisfy $\iota_1(\real{2}) = \real{2}= T_0S_1(A)$ and $\iota_2(\real{2})=M(A)= T_0S_2(A)$. Suppose that $S(A)$ is Lagrangian with respect to some K\"ahler form $\omega$. This implies in particular that the restriction of $\omega(0)$ to the two tangent spaces of $S(A)$ at the origin vanishes or, equivalently, $\iota_1^*\omega = 0$ and $\iota_2^*\omega$ = 0, where $\iota_1^*, \iota_2^*$ denote the respective pullbacks. Since $\omega$ is K\"ahler, we can write
\begin{equation*}\label{eq:32}
\omega = h_1 dz_1 \wedge d\close{z}_1 + \close{h} dz_1 \wedge d\close{z}_2 + h dz_2 \wedge d\close{z}_1 + h_2 dz_2 \wedge d\close{z}_2,
\end{equation*}
where at each point $p \in \compx{2}$, $H_\omega(p):=\begin{bmatrix}
h_1(p)&\close{h}(p)\\
h(p)&h_2(p)
\end{bmatrix}$ is a positive definite Hermitian matrix. The rest of our analysis takes place at the origin, and for simplicity, we shall use the notations $H_\omega:=H_\omega(0), h_1:=h_1(0)$, etc. Direct calculations give
\begin{equation}\label{eq:32}
\iota_1^*\omega = (\close{h}-h) dt \wedge ds,
\end{equation}
and
\begin{equation}\label{eq:33}
\iota_2^*\omega = \{(\close{h}-h)(xw-yz+1) + i[2h_1y - 2h_2z  -(\close{h}+h)(x-w) ]\}dt \wedge ds.
\end{equation}
The condition $\iota_1^*\omega = 0$ implies $h=\close{h} \in \real{}$ so $H_\omega$ has only real entries. Then (\ref{eq:33}) becomes
\begin{equation*}\label{eq:34}
\iota_2^*\omega = 2i(-hx+hw+h_1y-h_2z)dt \wedge ds.
\end{equation*}
and $\iota_2^*\omega=0$ gives
\begin{equation}\label{eq:35}
hx-hw-h_1y+h_2z = 0.
\end{equation}
We just showed that, if $A \in \mathcal{W}$ and $\omega$ is a K\"ahler form in $\compx{2}$ such that $S(A)$ is Lagrangian with respect to $\omega$ then, at the origin, the Hermitian matrix $H_\omega$ associated with $\omega$ is in fact a positive definite symmetric matrix with real entries and equation (\ref{eq:35}) has to be satisfied.

Now, let $\td{A}=\begin{bmatrix}
1&1\\
-1&2
\end{bmatrix}$. It is easy to see that $\td{A} \in \mathcal{W}$, hence $S(\td{A})$ is rationally convex. Suppose that $S(\td{A})$ is Lagrangian with respect to some K\"ahler form $\omega$ in $\compx{2}$. Then, $\td{A}$ has to satisfy equation (\ref{eq:35}). Substituting the entries of $\td{A}$ in (\ref{eq:35}) we obtain $h= -(h_1+h_2)$. Recall that $H_\omega=\begin{bmatrix}
h_1&h\\
h&h_2
\end{bmatrix}$ is positive definite which, by Sylvester's criterion, is equivalent to $H_\omega$ satisfying $h_1>0$ and $\det H_\omega >0$, which also implies that $h_2>0$. However, since $h= -(h_1+h_2)$, we have $\det H_\omega = h_1h_2-h^2= -h_1^2-h_2^2-h_1h_2 <0$ which is a contradiction. It follows that $S(\td{A})$ is not Lagrangian with respect to any K\"ahler form $\omega$. Note that by Theorem~\ref{th1} there exist a family of degenerate K\"ahler forms with respect to which $S(\td{A})$ is indeed isotropic. Furthermore, the following holds.
 
\begin{proposition}\label{prp:3}
The set of matrices $A \in \mathcal{W}$ such that $S(A)$ is not Lagrangian with respect to any K\"ahler form defined on $\compx{2}$ contains a nonempty subset which is open in $\mathcal{M}_{2 \times 2}(\real{})$.
\end{proposition}
\begin{proof}
Define the set
\begin{equation*}
\td{\mathcal{W}}:= \left\{\begin{bmatrix}
x&y\\
z&w
\end{bmatrix} \in \mathcal{W} : x-w \neq 0   \right\}.
\end{equation*}
Since $\mathcal{W}$ is open in $\mathcal{M}_{2 \times 2}(\real{})$ it is immediate that $\td{\mathcal{W}}$ is also open in $\mathcal{M}_{2 \times 2}(\real{})$. $\td{A} \in \td{\mathcal{W}}$, where $\td{A}$ is the matrix we defined in the example above, hence $\td{\mathcal{W}} \neq \emptyset$. Define
\begin{equation*}\label{eq:37}
F:\td{\mathcal{W}} \ni \begin{bmatrix}
x&y\\
z&w
\end{bmatrix} \mapsto x^2+w^2-2xw+2yz \in \real{},
\end{equation*}
which is a continuous function. Therefore, $F^{-1}((-\infty, 0))$ is open in $\mathcal{M}_{2 \times 2}(\real{})$ and it is nonempty because $\td{A} \in F^{-1}((-\infty, 0))$. Let $A:=\begin{bmatrix}
x&y\\
z&w
\end{bmatrix} \in F^{-1}((-\infty, 0))$ and suppose that $S(A)$ is Lagrangian with respect to some K\"ahler form $\omega$. Then, (\ref{eq:35}) is satisfied: $h(x-w)-h_1y+h_2z = 0$ and, since $x-w \neq 0$,
\begin{equation*}\label{eq:38}
h= \frac{h_1y-h_2z}{x-w}.
\end{equation*}
A direct computation gives
\begin{equation}\label{eq:39}
\begin{split}
\det H_\omega = h_1h_2 - h^2 &= \frac{h_1h_2(x^2+w^2-2xw+2yz) - h_1^2y^2 - h_2^2z^2}{(x-w)^2}\\ &= \frac{h_1h_2F(A) - h_1^2y^2 - h_2^2z^2}{(x-w)^2}.
\end{split}
\end{equation}
But $h_1h_2 >0, h_1^2y^2\geq 0, h_2^2z^2 \geq 0$ and $F(A)<0$ by construction, hence $\det H_\omega <0$ which is a contradiction. It follows that for every element $A$ of the nonempty open set $F^{-1}((-\infty, 0))$, $S(A)$ cannot be Lagrangian with respect to any K\"ahler form on $\compx{2}$, which ends the proof.
\end{proof}

\subsection{An Immersion in $\compx{2}$ which is not Locally Polynomially Convex}
Consider the matrix
\begin{equation*}
A = \begin{bmatrix}
0&-t\\
t&0
\end{bmatrix},
\end{equation*}
where $t>1$. Maintaining the notations from the previous subsection, let $S := M(A) \cup \real{2} \subset \compx{2}$. By a result of Weinstock \cite[Theorem 4]{4}, there exits a continuous one-parameter family of analytic annuli in $\compx{2}$ whose boundaries lie in $S$ and that converges to the origin as the parameter approaches $0$. Let $K := S \cap \close{B(0, \delta)}$ for some $\delta >0$. Clearly $K$ is not polynomially convex since, every annulus attached to $K$ is included in $\wh{K}$ by the maximum principle, hence $S$ is not locally polynomially convex at the origin. If $p$ is a point in the interior of one of the annuli attached to $K$ and $V$ is an algebraic variety passing through $p$, then it is known that $V$ should intersect all the members of the family. This implies that $V$ would intersect either the boundaries of the annuli or it would pass through the origin. Either way, $V$ would intersect $K$ which proves that $p \in \rhull{K}$, hence $K$ is not rationally convex.

\end{document}